\documentclass[12pt,twoside]{amsart}
\usepackage{amsmath}
\usepackage{amsthm}
\usepackage{amsfonts}
\usepackage{amssymb}
\usepackage{latexsym}
\usepackage[all]{xy}
\usepackage{epsfig}
\usepackage{amscd,mathrsfs,graphicx,mathrsfs}

\date{}
\allowdisplaybreaks[4] \footskip=15pt

\renewcommand{\uppercasenonmath}[1]{}

\numberwithin{equation}{section} \theoremstyle{plain}
\newtheorem*{theorem*}{Theorem A}
\newtheorem*{theorem**}{Theorem B}
\newtheorem*{theorem***}{Theorem C}
\newtheorem{theorem}{Theorem}[section]

\newtheorem*{corollary*}{Corollary}
\newtheorem{lemma}[theorem]{Lemma}
\newtheorem*{lemma*}{Lemma}
\newtheorem{proposition}[theorem]{Proposition}
\newtheorem*{proposition*}{Proposition}
\newtheorem{remark}[theorem]{Remark}
\newtheorem*{remark*}{Remark}

\newtheorem*{example*}{Example}
\newtheorem{definition}[theorem]{Definition}
\newtheorem*{definition*}{Definition}

\newtheorem*{conjecture*}{Conjecture}
\newtheorem*{ack*}{ACKNOWLEDGEMENTS}

\newcommand{\pf}{\noindent\begin {proof}}
\newcommand{\epf}{\end{proof}}

\begin{document}
\begin{center}
{\Large \bf On pure derived and pure singularity categories
 \footnotetext{
 $^{*}$ Corresponding author. E-mail address: wren@cqnu.edu.cn}}

\vspace{0.5cm}   Tianya Cao$^\dag$, Wei Ren$^*$\\
{\small $^\dag$ Department of Mathematics, Northwest Normal University, Lanzhou $730070$, China\\
$^*$ School of Mathematical Sciences, Chongqing Normal University, Chongqing $401331$, China}
\end{center}


\bigskip
\centerline { \bf  Abstract}
\leftskip10truemm \rightskip10truemm
\noindent Firstly, we compare the bounded derived categories with respect to the pure-exact and the usual exact structures,
and describe bounded derived category by pure-projective modules, under a fairly strong assumption on the ring. Then, we study Verdier quotient
of bounded pure derived category modulo the bounded homotopy category of pure-projective modules, which is called a pure singularity category since
we show that it reflects the finiteness of pure-global dimension of rings. Moreover, invariance of pure singularity in a recollement of bounded
pure derived categories is studied.
\bigskip


\leftskip0truemm \rightskip0truemm \vbox to 0.2cm{}

\section { \bf Introduction}
Recall that a short exact sequence $\mathbf{E}=0\rightarrow M_1\rightarrow M_2\rightarrow M_3\rightarrow 0$ of left $R$-modules is pure-exact provided that for any right $R$-module $N$ the induced sequence $N\otimes_{R}\mathbf{E}$ is exact. Purity is the basis of a relative homological theory, which leads to concepts such as pure-projective modules, pure-projective resolutions, pure derived functors Pext(-,-), pure-projective dimension etc., see for example \cite{Emm16, Gri70, JL89, KS75, PR05, Sim02, War69}.  It is also worth noting that a proof of the flat cover conjecture may be obtained by studying filtrations of a module by pure submodules \cite{BBE01}.

It is well known that in the sense of Neeman \cite{Nee90}, $\mathbf{D}(R)$ is the derived category of the exact category $(R\text{-Mod}, \mathcal{E})$, where $\mathcal{E}$ is the collection of all short exact sequences of $R$-modules. Let $\mathcal{E}_{pur}$ be the collection of all short pure-exact sequences, then $(R\text{-Mod}, \mathcal{E}_{pur})$ is an exact category. Recently, Zheng and Huang introduced in \cite{ZH16} the pure derived category $\mathbf{D}_{pur}(R)$ as the derived category of the exact category $(R\text{-Mod}, \mathcal{E}_{pur})$.

In this paper, we firstly intend to compare the bounded derived categories of these two exact structures, and to describe bounded derived category by pure-projective modules. Our first main result (Theorem \ref{thm 3.1}) shows that: If the subcategory $\mathcal{PP}$ of pure-projective $R$-modules is closed under kernels of epimorphisms, then there are triangle-equivalences:
\begin{center}$\mathbf{D}^{b}(R)\simeq\mathbf{D}_{pur}^{b}(R)/\mathbf{K}^{b}_{ac}(\mathcal{PP})\simeq \mathbf{K}^{-,ppb}(\mathcal{PP})/\mathbf{K}^{b}_{ac}(\mathcal{PP}),$\end{center}
where $\mathbf{K}^{-,ppb}(\mathcal{PP})$ is the homotopy category of upper bounded complexes of pure-projective modules which is bounded in the sense of pure cohomology, and $\mathbf{K}^{b}_{ac}(\mathcal{PP})$ is the homotopy category of bounded and exact (=acyclic) complexes of pure-projective modules. Some conditions for the assumption that $\mathcal{PP}$ is closed under kernels of epimorphisms are given in Remark \ref{rem 3.1}.

Recall that for a noetherian ring $R$, the category $\mathbf{D}_{sg}(R)$ is defined to be a Verdier quotient of bounded derived category of finitely generated modules $\mathbf{D}^{b}(R\text{-mod})$ modulo the bounded homotopy category of finitely generated projective modules $\mathbf{K}^{b}(R\text{-proj})$; it was also studied by Buchweitz \cite{Buc87} under the name of ``stable derived category''.
Since $\mathbf{D}_{sg}(R) = 0$ if and only if $R$ has finite global dimension (i.e. $\mathbf{D}_{sg}(R)$ reflects homological singularity of the ring $R$), this quotient triangulated category is called the singularity category of $R$ after Orlov \cite{Orl04}.
For any ring $R$, Beligiannis \cite{Bel00} studied the singularity category $\mathbf{D}^{b}(R)/\mathbf{K}^{b}(\mathcal{P})$, where $\mathbf{D}^{b}(R)$ is the bounded derived category of any $R$-modules, and $\mathbf{K}^{b}(\mathcal{P})$ is the bounded homotopy category of any projective $R$-modules. See also the literature \cite{Z15}.

We are inspired to study the pure version of singularity category. The pure singularity category $\mathbf{D}_{psg}(R)$ is defined
(Definition \ref{def 4.1}) as Verdier quotient $\mathbf{D}^{b}_{pur}(R)/\mathbf{K}^{b}(\mathcal{PP})$. We remark that the notion of pure singularity category seems to be reasonable since  $\mathbf{D}_{psg}(R)=0$ if and only if the pure-global dimension of $R$ is finite, see Proposition \ref{prop 4.2}. Moreover, we show that for any rings $A$, $B$ and $C$, if $\mathbf{D}^{b}_{pur}(A)$ admits a recollement
$\xymatrix{
   \mathbf{D}^{b}_{pur}(B)\ar[r]^{} & \mathbf{D}^{b}_{pur}(A) \ar[r]^{}\ar@<-0.8ex>[l]^{} \ar@<0.8ex>[l]_{} &\mathbf{D}^{b}_{pur}(C)\ar@<-0.8ex>[l]^{} \ar@<0.8ex>[l]_{}
   }$
relative to $\mathbf{D}^{b}_{pur}(B)$ and  $\mathbf{D}^{b}_{pur}(C)$, then $\mathbf{D}_{psg}(A)=0$ if and only if $\mathbf{D}_{psg}(B) = 0 = \mathbf{D}_{psg}(C)$; see Theorem \ref{thm 4.1}.

\section{\bf Preliminary}

Throughout this paper, $R$ denotes a ring with unity, and $R$-Mod the category of left $R$-modules. A short exact sequence $\mathbf{E}=
0\rightarrow M_1\stackrel{f}\rightarrow M_2\stackrel{g}\rightarrow M_3\rightarrow 0$ in $R$-Mod is called pure-exact if for any right $R$-module $N$, the induced sequence $0\rightarrow M_1\otimes_{R}N\rightarrow M_2\otimes_{R}N\rightarrow M_3\otimes_{R}N\rightarrow 0$ remains exact. In this case, $f$ is called pure monic and $g$ is called pure epic. By Cohn's theorem (see \cite[Theorem 3.65]{Rot79}), the above sequence is pure-exact if and only if $0\rightarrow \mathrm{Hom}_{R}(F, M_1)\rightarrow \mathrm{Hom}_{R}(F, M_2)\rightarrow \mathrm{Hom}_{R}(F, M_3)\rightarrow 0$ is exact for any finitely presented $R$-module $F$. Moreover, it turns out that the above exact sequence is pure-exact if and only if the associated short exact sequence of Pontryagin duals $0\rightarrow \mathrm{Hom}_{\mathbb{Z}}(M_3, \mathbb{Q}/\mathbb{Z})\rightarrow \mathrm{Hom}_{\mathbb{Z}}(M_2, \mathbb{Q}/\mathbb{Z}) \rightarrow \mathrm{Hom}_{\mathbb{Z}}(M_1, \mathbb{Q}/\mathbb{Z})\rightarrow 0$ is split (see \cite[Theorem 6.4]{JL89}). It is worth noting that the sequence $\mathbf{E}$ is pure-exact if and only if every finite system of linear equations $y_i = \sum_{j\in J} r_{ij}x_{j}$ with $y_{i}\in M_1$, $r_{ij}\in R$, $i\in I$, which has a solution $(x_{j})\in M_{2}^{J}$ also has a solution in $M_{1}^{J}$, where $I$ and $J$ are finite index sets.

Recall that an exact structure on an abelian category $\mathcal{A}$ is a class $\mathcal{E}$ of kernel-cokernel pairs $A\stackrel{f}\hookrightarrow B\stackrel{g}\twoheadrightarrow C$ which satisfies some axioms, see for example \cite{Buh10}. In the kernel-cokernel pair $(f,g)$, $f$ is called an admissible monic and $g$ is called an admissible epic. Let $\mathcal{E}_{pur}$ denote the class of all short pure-exact sequences, then $(R\text{-}\mathrm{Mod}, \mathcal{E}_{pur})$ is an exact category.

Let $X=\cdots \rightarrow X^{i-1}\stackrel{d^{i-1}}\rightarrow X^{i}\stackrel{d^{i}}\rightarrow X^{i+1}\rightarrow\cdots$ be an exact complex of modules. Then $X$ is said to be pure-exact at degree $n$ if the short exact sequence $0\rightarrow K^{n}\rightarrow X^{n}\rightarrow K^{n+1}\rightarrow 0$ is pure-exact, where $K^{n}= \mathrm{Ker}d^{n}$. $X$ is called pure-exact (=pure-acyclic) if it is pure-exact at all degree $n$.

An $R$-module $P$ is called pure-projective if $\mathrm{Hom}_{R}(P,-)$ is exact on all pure-exact sequences. Thus projective modules and finitely presented modules are pure-projective. Warfield \cite{War69} showed that a module $P$ is pure-projective if and only if it is a direct summand of a direct sum of finitely presented modules. We use $\mathcal{P}$ and $\mathcal{PP}$ to denote the class of all projective modules and pure-projective modules, respectively.

Let $\mathcal{C}$ be a full subcategory of an additive category $\mathcal{A}$, and $X\in \mathcal{A}$. Recall that a morphism $f: C\rightarrow X$ with $C\in\mathcal{C}$ is a $\mathcal{C}$-precover (or, right $\mathcal{C}$-approximation) of $X$, if $\mathrm{Hom}_{\mathcal{A}}(C^{'}, C)\rightarrow\mathrm{Hom}_{\mathcal{A}}(C^{'}, X)$ is surjective for each $C^{'}\in\mathcal{C}$. If each object $X\in\mathcal{A}$ admits a  $\mathcal{C}$-precover, then $\mathcal{C}$ is said to be precovering (or, contravariantly finite) in $\mathcal{A}$.
Any module $M$ has a pure-projective precover,  for example, if $M$ is expressed as the colimit $\underrightarrow{\mathrm{lim}}M_{i}$ of a directed system of finitely presented modules, then the canonical map $\bigoplus_{i}M_{i}\rightarrow M$ is a pure epimorphism (a pure-projective precover) (see \cite[\S II.1.1.3]{RG71}); see also \cite{War69} for the existence of pure-projective precovers. Thus, $\mathcal{PP}$ is precovering, and any module $M$ has a pure-projective resolution, i.e.  a pure-exact complex $\cdots \rightarrow P_{1}\rightarrow P_0\rightarrow M\rightarrow 0$ with each $P_i$ pure-projective.

We always identify an $R$-module $M$ with the complex concentrated in degree 0. For a complex $X$ and an integer $n$, $X[n]$ denotes the complex $X$ shifting $n$ degree to the left, that is, $X[n]^{m} = X^{n+m}$ and $d_{X[n]}^{m} = (-1)^{n}d^{n+m}_{X}$.
Given two $R$-complexes $X$ and $Y$, the complex $\mathrm{Hom}_{R}(X, Y)$ is defined with
$\mathrm{Hom}_{R}(X, Y)^{n} = \prod_{k\in \mathbb{Z}}\mathrm{Hom}_{R}(X^{k}, Y^{k+n})$, and with differential
$d^{n}(f^{k}) = (d^{k+n}_{Y}f^{k} - (-1)^{n}f^{k+1}d^{k}_{X})_{k\in\mathbb{Z}}$ for
$f = (f^{k})\in \mathrm{Hom}_{R}(X, Y)^{n}$.

A cochain map $f: X\rightarrow Y$ of complexes is a family of morphisms $f = (f^{n}: X^{n}\rightarrow Y^{n})_{n\in \mathbb{Z}}$
of $R$-modules satisfying $d_{Y}^{n}f^{n} = f^{n+1}d^{n}_{X}$ for all $n\in \mathbb{Z}$. Cochain maps
$f, g: X\rightarrow Y$ are called homotopic, denoted $f\sim g$, if there exists a family of
morphisms $(s^{n}: X^{n}\rightarrow Y^{n-1})_{n\in\mathbb{Z}}$ of $R$-modules, satisfying
$f^{n}-g^{n} = d^{n-1}_{Y}s^{n} + s^{n+1}d^{n}_{X}$ for all $n\in \mathbb{Z}$. A map $f: X\rightarrow Y$ of complexes is called a quasi-isomorphism if it induces isomorphic homology groups, and $f$ is called a homotopy equivalence if there exists a cochain map $g: Y\rightarrow X$ such that $gf\sim \mathrm{Id}_{X}$ and $fg\sim \mathrm{Id}_{Y}$. For $\mathrm{Con}(f)$ we mean the mapping cone of $f$, which is defined with $\mathrm{Con}(f)^{n} =X^{n+1}\oplus Y^{n}$ and with differential $d^{n}_{\mathrm{Con}(f)}=\begin{pmatrix} -d^{n+1}_{X}&0\\ f^{n+1}&d^{n}_{Y}\end{pmatrix}$
for all $n\in\mathbb{Z}$. Note that $f: X\rightarrow Y$ is a quasi-isomorphism if and only if $\mathrm{Con}(f)$ is an exact complex.

Let $X$ be an $R$-complex. It follows immediately that $X$ is pure-exact if $\mathrm{Hom}_{R}(P, X)$ is exact for each pure-projective module $P$. A cochain map $f: X\rightarrow Y$ of $R$-complexes is a pure-quasi-isomorphism provided that $\mathrm{Hom}_{R}(P, f)$ is a quasi-isomorphism for all $P\in \mathcal{PP}$, or equivalently, $\mathrm{Con}(f)$ is a pure-exact complex. It is easy to see that every pure-exact complex is exact, and every pure-quasi-isomorphism is a quasi-isomorphism.

It is well known that the derived category is a Verdier quotient of the homotopy category with respect to the thick triangulated subcategory of
exact complexes. In general, given a triangulated subcategory $\mathcal{B}$ of a triangulated category $\mathcal{K}$, in the Verdier quotient
$\mathcal{K} / \mathcal{B} = S^{-1}\mathcal{K}$, where $S$ is the compatible multiplicative system determined by $\mathcal{B}$, each morphism
$f: X\rightarrow Y$ is given by an equivalent class of right fractions $a/s$ presented by $X\stackrel{s}\Longleftarrow Z \stackrel{a}\longrightarrow Y$.

We use ``$\cong$'' to denote the isomorphsims of objects, and ``$\simeq$'' to denote the equivalences of categories.
As usual, $\mathbf{K}(R)$ is the homotopy category of $R$-complexes, and $\mathbf{D}(R)$ is the derived category of $R$; and furthermore, we use the superscript $\ast\in \{-, +, b\}$ to denote the corresponding subcategories with conditions of bounded above, bounded below and bounded, respectively. We denote by $\mathbf{K}^{\ast}_{ac}(R)$ and $\mathbf{K}_{pac}^{\ast}(R)$ the homotopy category consisting of exact (acyclic) complexes and pure-exact (pure-acyclic) complexes, respectively. Clearly, $\mathbf{K}_{ac}^{\ast}(R)$ is a thick triangulated subcategory of $\mathbf{K}^{\ast}(R)$. Since pure-acyclic complexes are closed under summands, it follows from Rickard's criterion (see \cite[Proposition 1.3]{Ric89} or \cite[Criterion 1.3]{Nee90}) that $\mathbf{K}^{\ast}_{pac}(R)$ is a thick triangulated subcategory of $\mathbf{K}_{ac}^{\ast}(R)$, and hence of $\mathbf{K}^{\ast}(R)$. We remark that the terminology ''thick'' in \cite{Ric89} is ``\'{e}paisse'' in French. For $\ast\in \{\text{blank}, -, b\}$, let $\mathbf{K}^{\ast}(\mathcal{P})$ and $\mathbf{K}^{\ast}(\mathcal{PP})$ be respectively the homotopy category of complexes of projective modules and pure-projective modules.

\section{\bf Pure derived categories and derived categories}
Zheng and Huang \cite{ZH16} defined the pure derived category $\mathbf{D}^{\ast}_{pur}(R):= \mathbf{K}^{\ast}(R)/\mathbf{K}^{\ast}_{pac}(R)$ as a Verdier quotient of $\mathbf{K}^{\ast}(R)$ modulo the thick subcategory $\mathbf{K}^{\ast}_{pac}(R)$, where $\ast\in \{\text{blank}, -, b\}$.
In fact, $\mathbf{D}^{\ast}_{pur}(R)$ is the derived category of the exact categories ($R$-Mod, $\mathcal{E}_{pur}$) in the sense of Neeman \cite{Nee90}, where $\mathcal{E}_{pur}$ is the collection of all short pure-exact sequences of $R$-modules.

In this section, we intend to compare $\mathbf{D}^{b}(R)$ and $\mathbf{D}_{pur}^{b}(R)$, and to describe $\mathbf{D}^{b}(R)$ by pure-projective modules. The main result of this section is stated below, where
$$\mathbf{K}^{-,ppb}(\mathcal{PP}):= \left\{X\in \mathbf{K}^{-}(\mathcal{PP}) \, \vline \,\begin{matrix}  \text{there exists }
n=n(X)\in\mathbb{Z}, \text{ such that }\\
\mathrm{H}^{i}\mathrm{Hom}_{R}(P,X) = 0, \forall i\leq n, \forall P\in \mathcal{PP}
 \end{matrix} \right\},$$
and $\mathbf{K}^{b}_{ac}(\mathcal{PP})$ is the homotopy category of bounded exact complexes of pure-projective modules.

\begin{theorem}\label{thm 3.1} Let $R$ be a ring. If the subcategory $\mathcal{PP}$ of pure-projective $R$-modules is closed under kernels of epimorphisms, then there are triangle-equivalences
$$\mathbf{D}^{b}(R)\simeq\mathbf{D}_{pur}^{b}(R)/\mathbf{K}^{b}_{ac}(\mathcal{PP})\simeq \mathbf{K}^{-,ppb}(\mathcal{PP})/\mathbf{K}^{b}_{ac}(\mathcal{PP}).$$
\end{theorem}

Next, some examples of rings such that the subcategory $\mathcal{PP}$ of pure-projective $R$-modules is closed under kernels of epimorphisms will be given. However, due to the following counterexamples suggested by M. Hrbek, the assumption is fairly strong. For example, if the ring $R$ is not coherent, then for any finitely presented module $F$, its finitely generated, by not finitely presented, submodule $G$ will give a counterexample, as the epimorphism $F\rightarrow F/G$ has the kernel $G$ which is not pure-projective. It follows from \cite{Bru79} that for an artin algebra $A$, maximal submodules of pure-projective modules are pure-projective if and only if $A$ is of finite representation type, and then there are counterexamples of the above assumption for artin algebras of non-finite representation type. We look forward to weakening this assumption.

\begin{remark}\label{rem 3.1}
\indent$(1)$ It follows immediately from \cite[Corollary 4.3]{Ver97} that for $\ast \in \{{blank, -, +, b}\}$, there is an equivalence of triangulated categories $\mathbf{D}^{\ast}(R)\simeq \mathbf{D}^{\ast}_{pur}(R)/(\mathbf{K}^{\ast}_{ac}(R)/ \mathbf{K}^{\ast}_{pac}(R))$.
Moreover, it is easy to see that $\mathbf{D}^{\ast}(R)\simeq \mathbf{D}^{\ast}_{pur}(R)$ if and only if $\mathbf{K}_{ac}^{\ast}(R)\simeq \mathbf{K}^{\ast}_{pac}(R)$, if and only if $\mathcal{PP}=\mathcal{P}$. It is clear that when $R$ is von Neumann regular, these conditions hold.\\
\indent$(2)$ Let $R$ be a commutative ring. It follows from \cite[Theorem 4.3]{Gri70} that every $R$-module is pure-projective if and only if $R$ is  an artinian principal ideal ring, if and only if $R$ is a generalized uniserial quasi-Frobenius ring. In this case $\mathcal{PP}$ is obviously closed under kernels of epimorphisms. The rings with the above properties are also called pure-semisimple, and are characterized in
\cite[Theorem 8.4]{JL89}.\\
\indent$(3)$ Recall that Kulikov proved in 1945 that subgroups of finitely generated abelian groups are again direct sums of finitely generated groups. It follows from a version of Kulikov's theorem \cite[Theorem 2.1]{Bru83} that there are some rings with Kulikov property (the heredity of pure-projectivity), i.e. any submodule of pure-projective module is pure-projective. For example, consider the oriented cycle $\xymatrix{\Gamma_{n}: 1\ar[r] &2\ar[r] &\cdots \ar[r] &n \ar@/_1pc/[lll]}$ of length $n\geq 1$, and then the path algebra $k\Gamma_{n}$ of this quiver over some field $k$ admits Kulikov property as the category of $k\Gamma_{n}$-modules is isomorphic to the category of $k$-linear representations of $\Gamma_{n}$; the $k$-linearization $k[\mathbb{N}]$ of the ordered set $\mathbb{N}$ has Kulikov property; let $R$ be a Dedekind domain which is not a field, then $R$ has Kulikov property;  see \cite[pp. 34]{Bru83} and \cite[Example 11.19]{JL89}.
\end{remark}

In order to prove Theorem \ref{thm 3.1}, we need to make some preparations. Let $\mathbf{K}^{-,b}(\mathcal{P})$ be the homotopy category of upper bounded complexes of projective modules with only finitely many non-zero cohomologies.

\begin{lemma}\label{lem 3.1}
Let $X\in \mathbf{K}^{-,b}(\mathcal{P})$. Then there exists a quasi-isomorphism $X\rightarrow P$ with $P\in\mathbf{K}^{-,ppb}(\mathcal{PP})$.
\end{lemma}

\begin{proof}
For $X\in\mathbf{K}^{-,b}(\mathcal{P})$, there is an integer $n\in\mathbb{Z}$ such that $\mathrm{H}^{i}(X)=0$ for $i\leq n$. For $\mathrm{Ker}d_{X}^{n}$ there is a pure-projective resolution
$\cdots\rightarrow P^{n-2}\rightarrow P^{n-1}\rightarrow \mathrm{Ker}d_{X}^{n}\rightarrow 0$.
By a version of comparison theorem, we can get a cochain map
$$\xymatrix{
X=\ar[d]_{f}& \cdots \ar[r] &X^{n-2} \ar[r]\ar[d] & X^{n-1} \ar[r] \ar[d] &X^{n} \ar[r]^{d_{X}^{n}}\ar[d]_{\|} & X^{n+1} \ar[r]\ar[d]_{\|} &\cdots\\
P=& \cdots \ar[r] &P^{n-2}\ar[r] &P^{n-1}\ar[r] &X^{n}\ar[r] &X^{n+1}\ar[r] & \cdots
}$$
From the construction, it is direct that $f$ is a quasi-isomorphism and $P\in\mathbf{K}^{-,ppb}(\mathcal{PP})$.
\end{proof}

\begin{lemma}\label{lem 3.2}
Let $X\in \mathbf{K}^{-,b}(\mathcal{P})$ and $G\in \mathbf{K}^{-,ppb}(\mathcal{PP})$. Then for any cochain map $g: X \rightarrow G$, there is a quasi-isomorphism $f: X\rightarrow P$ with $P\in\mathbf{K}^{-,ppb}(\mathcal{PP})$, and a cochain map $h: P\rightarrow G$, such that $g\sim hf$.
\end{lemma}

\begin{proof}
For $X\in \mathbf{K}^{-,b}(\mathcal{P})$ and $G\in \mathbf{K}^{-,ppb}(\mathcal{PP})$, there exist integers $n(X)$ and $n(G)$, such that
$\mathrm{H}^{i}\mathrm{Hom}_{R}(M,X)=0$ for any $i\leq n(X)$ and any $M\in \mathcal{P}$, and $\mathrm{H}^{j}\mathrm{Hom}_{R}(N,G)=0$ for any $j\leq n(G)$ and any $N\in \mathcal{PP}$. Let $n= \mathrm{min}\{n(X), n(G)\}$. By the construction in the above lemma, we get a quasi-isomorphism $f: X\rightarrow P$, where $P\in\mathbf{K}^{-,ppb}(\mathcal{PP})$ with $P^i=X^i$ for $i\geq n$.

For any $i\geq n$, $f^{i}= \mathrm{Id}_{X^{i}}$, and let $h^i = g^i$. Since $G\in \mathbf{K}^{-,ppb}(\mathcal{PP})$, the sequence
$$0\longrightarrow \mathrm{Ker}d_{G}^{n-1}\longrightarrow G^{n-1}\longrightarrow \mathrm{Ker}d_{G}^{n}\longrightarrow 0$$ is pure-exact, and then it remains exact by applying $\mathrm{Hom}_{R}(P^{n-1},-)$. It follows from $$d_{G}^{n}h^{n}d_{P}^{n-1} = d_{G}^{n}g^{n}d_{P}^{n-1}= g^{n+1}d_{P}^{n}d_{P}^{n-1} = 0$$ that
$h^{n}d_{P}^{n-1}\in \mathrm{Hom}_{R}(P^{n-1}, \mathrm{Ker}d_{G}^{n})$. This yields a morphism $h^{n-1}\in \mathrm{Hom}_{R}(P^{n-1}, G^{n-1})$
such that $d_{G}^{n-1}h^{n-1} = h^{n}d_{P}^{n-1}$. Inductive, we get morphisms $h^{j}: P^{j}\rightarrow G^{j}$ ($j<n$) such that $d_{G}^{j-1}h^{j-1} = h^{j}d_{P}^{j-1}$. Hence we have a cochain map $h: P\rightarrow G$.

Now consider the following diagram:
$$\xymatrix@R=20pt@C=10pt{
X= \cdots \ar[r] &X^{n-2}\ar[rr]\ar[ddd]_{g^{n-2}}\ar[rd]^{f^{n-2}} &&X^{n-1}\ar[rr]\ar[ddd]_{g^{n-1}}\ar[rd]^{f^{n-1}} &&X^{n}\ar[rr]\ar[ddd]_{g^{n}}\ar[rd]^{\mathrm{Id}} &&X^{n+1}\ar[r]\ar[ddd]_{g^{n+1}}\ar[rd]^{\mathrm{Id}} &\cdots\\
P= \cdots \ar[rr] &&P^{n-2}\ar[rr]\ar@{-->}[ldd]^{h^{n-2}} &&P^{n-1}\ar[rr]\ar@{-->}[ldd]^{h^{n-1}} &&X^{n}\ar[rr]\ar[ldd]^{h^{n}} &&X^{n+1}\ar[r]\ar[ldd]^{h^{n+1}} &\cdots\\
\\
G= \cdots \ar[r] &G^{n-2}\ar[rr] &&G^{n-1}\ar[rr] &&G^{n}\ar[rr] &&G^{n+1}\ar[r] &\cdots
}$$
It is easy to see that $g^{i} - h^{i}f^{i}=0$ for $i\geq n$. Since $G$ is exact at the degree less than $n$, the sequence
$0\rightarrow \mathrm{Ker}d_{G}^{n-2}\rightarrow G^{n-2}\rightarrow \mathrm{Ker}d_{G}^{n-1}\rightarrow 0$ is exact. Note that
$d_{G}^{n-1}(g^{n-1}- h^{n-1}f^{n-1})=0$, then $g^{n-1}- h^{n-1}f^{n-1}\in \mathrm{Hom}_{R}(X^{n-1}, \mathrm{Ker}d_{G}^{n-1})$.
Since $X^{n-1}$ is projective, there exists a map $s^{n-1}: X^{n-1}\rightarrow G^{n-2}$ such that $g^{n-1}- h^{n-1}f^{n-1} = d_{G}^{n-2}s^{n-1}$. By induction we get homotpy maps $s: X\rightarrow G[-1]$ with $s^{i}=0$ for any $i\geq n$. Hence $g-hf: X\rightarrow G$ is null homotopy. This completes the proof.
\end{proof}

\begin{lemma}\label{lem 3.3}
Let $P\in \mathbf{K}^{-}(\mathcal{PP})$. If $P$ is pure-exact, then $P = 0$ in $\mathbf{K}(R)$.
\end{lemma}

\begin{proof}
Without loss of generality, we assume that
$$P = \cdots \longrightarrow P^{-2}\stackrel{d^{-2}}\longrightarrow P^{-1}\stackrel{d^{-1}}\longrightarrow P^0\longrightarrow 0.$$
Set $P^0 = \mathrm{Ker}d^0$, and consider the sequences $0\rightarrow \mathrm{Ker}d^{i-1}\rightarrow P^{i-1}\rightarrow \mathrm{Ker}d^{i}\rightarrow 0$, $i\leq 0$. These sequences are pure-exact with each module pure-projective, and hence they are split. This yields that
$P^{i}\cong \mathrm{Ker}d^i \oplus \mathrm{Ker}d^{i+1}$, and $P$ is a direct sum of contractible complexes of the form
$\cdots \rightarrow 0\rightarrow \mathrm{Ker}d^{i}\stackrel{=}\rightarrow \mathrm{Ker}d^{i}\rightarrow 0\rightarrow \cdots$.
Thus $P = 0$ in $\mathbf{K}(R)$.
\end{proof}

\begin{lemma}\label{lem 3.4}
Let $P\in \mathbf{K}^{-,ppb}(\mathcal{PP})$. If $\mathcal{PP}$ is closed under kernels of epimorphisms and $P$ is acyclic, then $P\in\mathbf{K}^{b}_{ac}(\mathcal{PP})$.
\end{lemma}

\begin{proof}
For $P\in \mathbf{K}^{-,ppb}(\mathcal{PP})$, there exists an integer $n\in\mathbb{Z}$ such that $\mathrm{H}^{i}\mathrm{Hom}_{R}(Q, P)=0$ for any $i\leq n$ and any pure-projective module $Q$. Let
$$P^{'}= \cdots\longrightarrow 0\longrightarrow 0\longrightarrow \mathrm{Im}d^{n}\longrightarrow P^{n+1}\longrightarrow P^{n+2}\longrightarrow \cdots,$$
$$P^{''}=  \cdots\longrightarrow P^{n-2}\longrightarrow P^{n-1}\longrightarrow \mathrm{Ker}d^{n}\longrightarrow 0\rightarrow 0\longrightarrow \cdots.$$
We have an exact sequence of complexes $0\rightarrow P^{''}\stackrel{f}\rightarrow P \stackrel{g}\rightarrow P^{'}\rightarrow 0$. By the assumption,  $\mathcal{PP}$ is closed under kernels of epimorphisms, and $P$ is acyclic and upper bounded, then $\mathrm{Im}d^{n}$ and $\mathrm{Ker}d^{n}$ are pure-projective. Moreover, $\mathrm{H}^{n}\mathrm{Hom}_{R}(Q, P)=0$ for any pure-projective module $Q$. This implies that the sequence $0\rightarrow \mathrm{Ker}d^{n} \rightarrow P^n\rightarrow \mathrm{Im}d^{n}\rightarrow 0$ is pure-exact, and moreover, it is split. Thus, the sequence of complexes  $0\rightarrow P^{''}\rightarrow P\rightarrow P^{'}\rightarrow 0$ is split degree-wise, and then there is a distinguished triangle  in the homotopy category $\mathbf{K}(R)$:
$$P^{''}\stackrel{f}\longrightarrow P \stackrel{g}\longrightarrow P^{'}\stackrel{h}\longrightarrow P^{''}[1].$$
Note that $h=0$, so the triangle is split and $P= P^{'} \oplus P^{''}$.  Since $P\in \mathbf{K}^{-,ppb}(\mathcal{PP})$, we have $P^{''}\in \mathbf{K}^{-}(\mathcal{PP})$; moreover, $P^{''}$ is pure-exact, and then $P^{''} = 0$ by Lemma \ref{lem 3.3}. Hence $P\cong P^{'}\in\mathbf{K}^{b}_{ac}(\mathcal{PP})$.
\end{proof}

Now we are in a position to prove the main result of this section.

\subsection*{Proof of Theorem \ref{thm 3.1}}
Let $\mathrm{Inc}:\mathbf{K}^{-,ppb}(\mathcal{PP})\longrightarrow \mathbf{K}^{-}(R)$ be the embedding functor, and $Q:\mathbf{K}^{-}(R)\rightarrow \mathbf{D}^{-}(R)$ be the canonical localization functor. We denote the composition functor by $\eta:\mathbf{K}^{-,ppb}(\mathcal{PP})\rightarrow \mathbf{D}^{-}(R)$. Since $\eta(\mathbf{K}^{b}_{ac}(\mathcal{PP}))=0$, by the universal property of quotient functor we have an unique triangle functor $$\overline{\eta}:\mathbf{K}^{-,ppb}(\mathcal{PP})/ \mathbf{K}^{b}_{ac}(\mathcal{PP})\rightarrow \mathbf{D}^{-}(R).$$
Clearly, $\mathrm{Im}\overline{\eta}\subseteq \mathbf{D}^{b}(R)$. By Lemma \ref{lem 3.1}, for any  $X\in \mathbf{K}^{-,b}(\mathcal{P})\simeq\mathbf{D}^{b}(R)$, there exists $P\in\mathbf{K}^{-,ppb}(\mathcal{PP})$ such that $X = \overline{\eta}(P)$.
Hence the triangle functor $\overline{\eta}:\mathbf{K}^{-,ppb}(\mathcal{PP})/
\mathbf{K}^{b}_{ac}(\mathcal{PP})\rightarrow \mathbf{D}^{b}(R)$ is dense.

Let $P_{1}, P_{2}\in\mathbf{K}^{-,ppb}(\mathcal{PP})$ and $\alpha/s: P_{1}\Longleftarrow Y\longrightarrow P_{2}$ be a morphism in $\mathbf{D}^{b}(R)$, where $s: Y\Longrightarrow P_{1}$ is a quasi-isomorphism with $Y\in\mathbf{K}^{-}(R)$ and $\alpha: Y\rightarrow P_{2}$ is a morphism in $\mathbf{K}^{-}(R)$. Let $t: X\rightarrow Y$ be a dg-projective resolution of $Y$, then $t$ is a quasi-isomorphism and we can let $X\in \mathbf{K}^{-,b}(\mathcal{P})$.

For $X\in \mathbf{K}^{-,b}(\mathcal{P})$ and $P_{1}, P_{2}\in \mathbf{K}^{-,ppb}(\mathcal{PP})$, there exist integers $n(X)$, $n(P_{1})$ and $n(P_{2})$, such that $\mathrm{H}^{i}\mathrm{Hom}_{R}(M,X)=0$ for any $i\leq n(X)$ and any $M\in \mathcal{P}$, and $\mathrm{H}^{j}\mathrm{Hom}_{R}(N, P_{k})=0$ for any $j\leq n(P_{k})$ and any $N\in \mathcal{PP}$ ($k=1,2$). Let $n= \mathrm{min}\{n(X), n(P_{1}), n(P_{2})\}$. We consider pure-projective resolution of $\mathrm{Ker}d_{X}^{n}$, and it follows from the above construction that there is a complex $P\in \mathbf{K}^{-,ppb}(\mathcal{PP})$ and a quasi-isomorphism $f: X\rightarrow P$, such that for morphisms
$st: X\rightarrow P_{1}$ and $\alpha t: X\rightarrow P_{2}$, we have morphisms
$g_1: P\rightarrow P_{1}$ and $g_{2}: P\rightarrow P_{2}$ satisfying $st \sim g_{1}f$ and $\alpha t \sim g_{2}f$. Then
we have the following commutative diagram in $\mathbf{K}^{-}(R)$:
$$\xymatrix{
& Y\ar@{=>}[ld]_{s} \ar[rd]^{\alpha}\\
P_{1} &X \ar@{=>}[l]_{st}\ar[r]^{\alpha t}\ar@{=>}[u]_{t}\ar@{=>}[d]^{f} &P_{2}\\
& P \ar@{=>}[lu]^{g_{1}}\ar[ru]_{g_{2}}
   }$$
where the double arrowed morphisms mean quasi-isomorphisms. Note that $g_{1}$ is also a quasi-isomorphism, hence the mapping cone $\mathrm{Con}(g_{1})$ is acyclic. Moreover, $\mathrm{Con}(g_{1})\in\mathbf{K}^{-,ppb}(\mathcal{PP})$, and it follows from Lemma \ref{lem 3.4} that $\mathrm{Con}(g_{1})\in\mathbf{K}^{b}_{ac}(\mathcal{PP})$. By the definition of right fraction, we have
$g_{2}/g_{1}\in\mathrm{Hom}_{\mathbf{K}^{-,ppb}(\mathcal{PP})/
\mathbf{K}^{b}_{ac}(\mathcal{PP})}(P_{1}, P_{2})$
and $\alpha/s = g_{2}/g_{1}=\overline{\eta}(g_{2}/g_{1})$. This implies that the functor $\overline{\eta}$ is full.

It remains to prove $\overline{\eta}$ is faithful. Because the triangle functor $\overline{\eta}$ is full, by \cite[p.446]{Ric89} it suffices
to show that it sends non-zero objects to non-zero objects. Suppose $P\in\mathbf{K}^{-,ppb}(\mathcal{PP})$ and $\overline{\eta}(P)=0$, then $P$ is acyclic and it follows from Lemma \ref{lem 3.4} that $P\in\mathbf{K}^{b}_{ac}(\mathcal{PP})$. Hence $\overline{\eta}$ is faithful. This completes the proof.
\hfill$\square$

\section{\bf Pure singularity categories}
By \cite[Theorem 3.6 (1)]{ZH16}, for any ring $R$ there is a triangle-equivalence $\mathbf{D}^{b}_{pur}(R)\simeq\mathbf{K}^{-,ppb}(\mathcal{PP})$, and it is trivial that $\mathbf{K}^{b}(\mathcal{PP})$ is a thick subcategory of $\mathbf{K}^{-,ppb}(\mathcal{PP})$. In this section, we introduce and study the pure singularity category, as a version of the singularity category with respect to pure-exact structure.  Note that relative singularity categories under another conditions have also been studied, see for example \cite{BDZ15, Chen11, LH15}.

\begin{definition}\label{def 4.1}
For a ring $R$, the pure singularity category is defined to be the Verdier quotient
$$\mathbf{D}_{psg}(R):= \mathbf{D}^{b}_{pur}(R)/\mathbf{K}^{b}(\mathcal{PP})\simeq
\mathbf{K}^{-,ppb}(\mathcal{PP}) / \mathbf{K}^{b}(\mathcal{PP}).$$
\end{definition}

Warfield \cite{War69} showed that any  module $M$ has a pure-projective resolution $\mathbf{P}\rightarrow M\rightarrow 0$. In \cite{Gri70} Griffith defined $\mathrm{Pext}_{R}^{i}(M, N):= \mathrm{H}^{i}\mathrm{Hom}_{R}(\mathbf{P}, N)$. Analogous to the setting of derived category, the following shows that the relative derived functors of Hom with respect to pure-projective modules can be interpreted as the morphisms in the corresponding relative derived category with respect to pure-projective modules.

\begin{proposition}\label{prop 4.1}
Let $M$, $N$ be any $R$-modules. Then $\mathrm{Pext}^{i}_{R}(M, N) \cong
\mathrm{Hom}_{\mathbf{D}^{b}_{pur}(R)}(M, N[i])$.
\end{proposition}

\begin{proof}
Let $\mathbf{P}\rightarrow M\rightarrow 0$ be a pure-projective resolution of $M$. Viewing $M$ as a complex concentrated in degree zero, then $\mathbf{P}\rightarrow M$ is a pure-quasi-isomorphism, and $\mathbf{P}\cong M$ in $\mathrm{\mathbf{D}}^{b}_{pur}(R)$. Hence we have
$$\begin{aligned}
\mathrm{Pext}^{i}_{R}(M, N) &= \mathrm{H}^{i}\,\mathrm{Hom}_{R}(\mathbf{P}, N)\\
&= \mathrm{Hom}_{\mathrm{\mathbf{K}}(R)}(\mathbf{P}, N[i])\\
&\cong \mathrm{Hom}_{\mathrm{\mathbf{D}}_{pur}(R)}(\mathbf{P}, N[i])\\
&\cong \mathrm{Hom}_{\mathrm{\mathbf{D}}^{b}_{pur}(R)}(M, N[i])
\end{aligned}$$
where the last two isomorphisms hold by \cite[Proposition 3.1(1)]{ZH16} and \cite[Proposition 3.2]{ZH16}.
\end{proof}

\begin{definition}\label{def 4.2}(\cite{Gri70})
Let $M$ be an $R$-module. The pure-projective dimension $\mathrm{p.pd}(M)$ of $M$ is defined to be the smallest positive integer $n$ such that $\mathrm{Pext}_{R}^{n+1}(M, N) = 0$, and set $\mathrm{p.pd}(M) = \infty$ if no such $n$ exists. For any ring $R$, the pure-global dimension $\mathrm{p.gldim}(R)$ is defined as the supremum of the pure-projective dimensions of all $R$-modules.
\end{definition}

Note that pure-projective dimension of a module $M$ can be also defined by the shortest length of pure-projective resolutions of $M$. It is a routine job to show that these two definitions of pure-projective dimension coincide. For any ring $R$, the supremum of the pure-projective dimensions of all $R$-modules is equal to the supremum of the pure-injective dimensions of all $R$-modules.

Recall that $\mathbf{D}_{sg}(R)=0$ if and only if the global dimension of $R$ is finite. We have the following pure version, which implies that the notion of ``pure singularity'' seems to be reasonable.

\begin{proposition}\label{prop 4.2}
Let $R$ be a ring. Then $\mathbf{D}_{psg}(R)=0$ if and only if $\mathrm{p.gldim}(R)$ is finite.
\end{proposition}

\begin{proof}
If the pure-global dimension $\mathrm{p.gldim}(R)$ of $R$ is finite, then it follows from \cite[Theorem 4.7]{ZH16} that for any $X\in \mathbf{D}^{b}_{pur}(R)$, there exists a pure-quasi-isomorphism $P\rightarrow X$ with $P\in \mathbf{K}^{b}(\mathcal{PP})$. Then $\mathbf{D}^{b}_{pur}(R)\simeq \mathbf{K}^{b}(\mathcal{PP})$, and hence $\mathbf{D}_{psg}(R)=0$.

It remains to prove the necessity. Suppose $\mathbf{D}_{psg}(R)=0$. Let $M$ be an $R$-module. Then $M=0$ in $\mathbf{D}_{psg}(R)$, and there exists some $X\in \mathbf{K}^{b}(\mathcal{PP})$ such that $M\cong X$ in $\mathbf{D}^{b}_{pur}(R)$. We denote this isomorphism by a right fraction $\alpha/f: M\stackrel{f}\Longleftarrow Y \stackrel{\alpha}\longrightarrow X$, where $f$ and $\alpha$ are pure-quasi-isomorphisms. Then there is a triangle $Y\stackrel{\alpha}\longrightarrow X\longrightarrow \mathrm{Con}(\alpha)\longrightarrow Y[1]$ in $\mathbf{K}(R)$ with $\mathrm{Con}(\alpha)$ pure-exact. By applying $\mathrm{Hom}_{\mathbf{K}(R)}(X,-)$ to it, we get an exact sequence
$$\mathrm{Hom}_{\mathbf{K}(R)}(X,Y)\longrightarrow \mathrm{Hom}_{\mathbf{K}(R)}(X,X)\longrightarrow \mathrm{Hom}_{\mathbf{K}(R)}(X,\mathrm{Con}(\alpha)).$$
It follows from \cite[Lemma 2.9(1)]{ZH16} that $\mathrm{Hom}_{\mathbf{K}(R)}(X,\mathrm{Con}(\alpha))=0$, so there is a cochain map $\beta: X\rightarrow Y$ such that $\alpha\beta$ is homotopic to $\mathrm{Id}_{X}$. Thus we get a pure-quasi-isomorphism $f\beta: X\rightarrow M$.

Consider the soft truncation $X_{0\supset}= \cdots\rightarrow X^{-2}\rightarrow X^{-1}\rightarrow \mathrm{Ker}d_{X}^{0}\rightarrow 0$ of $X$.
Then there is a pure-quasi-isomorphism $ f\beta\iota: X_{0\supset}\rightarrow M$, where $\iota: X_{0\supset}\rightarrow X$ is a natural embedding.
Since $X\in \mathbf{K}^{b}(\mathcal{PP})$, we assume that there is an integer $n$ such that $X^{i}=0$ for any $i>n$.  Note that the sequence $$0\longrightarrow \mathrm{Ker}d^{0}_{X}\longrightarrow X^{0}\longrightarrow\cdots\longrightarrow X^{n-1}\longrightarrow X^{n}\longrightarrow 0$$
is pure-exact with each $X^i\in \mathcal{PP}$, and it follows that $\mathrm{Ker}d^{0}_{X}$ is also pure-projective. Thus $M$ has a bounded pure-projective resolution $X_{0\supset}\rightarrow M$, and then $\mathrm{p.pd}(M)<\infty$. This yields the desired assertion.
\end{proof}

The notion of recollement of triangulated categories was introduced by Beilinson, Bernstein and Deligne \cite{BBD82} with an idea that one category can be viewed as being ``glued together'' from two others. Let $\mathcal{T}^{'}$,  $\mathcal{T}$ and  $\mathcal{T}^{''}$ be triangulated categories. We say that $\mathcal{T}$ admits a recollement relative to  $\mathcal{T}^{'}$ and  $\mathcal{T}^{'}$, if there exist six triangulated functors as in the following diagram
$$\xymatrix@C=40pt{
\mathcal{T}^{'} \ar[r]^{i_{*}} & \mathcal{T} \ar[r]^{j^{*}} \ar@/_1pc/[l]_{i^{*}}\ar@/^1pc/[l]^{i^{!}}
 & \mathcal{T}^{''}\ar@/_1pc/[l]_{j_{!}}\ar@/^1pc/[l]^{j_{*}}
}$$
such that \\
\indent $(1)$ $(i^{*}, i_{*})$, $(i_{*}, i^{!})$, $(j_{!}, j^{*})$ and $(j^{*}, j_{*})$ are adjoint pairs;  \\
\indent $(2)$ $i_{*}$, $j_{*}$ and $j_{!}$ are full embedding;\\
\indent $(3)$ $j^{*}i_{*} =0$;\\
\indent $(4)$ for each $X\in \mathcal{T}$, there are distinguished triangles
$$i_{*}i^{!}(X)\longrightarrow X\longrightarrow j_{*}j^{*}(X)\longrightarrow i_{*}i^{!}(X)[1],$$
$$j_{!}j^{*}(X)\longrightarrow X\longrightarrow i_{*}i^{*}(X)\longrightarrow j_{!}j^{*}(X)[1].$$

\begin{theorem}\label{thm 4.1}
Let $A$, $B$ and $C$ be rings. Assume that $\mathbf{D}^{b}_{pur}(A)$ admits the following recollement
$$\xymatrix@C=40pt{
\mathbf{D}^{b}_{pur}(B)\ar[r]^{i_{*}} & \mathbf{D}^{b}_{pur}(A) \ar[r]^{j^{*}} \ar@/_1pc/[l]_{i^{*}}\ar@/^1pc/[l]^{i^{!}}
 & \mathbf{D}^{b}_{pur}(C).\ar@/_1pc/[l]_{j_{!}}\ar@/^1pc/[l]^{j_{*}}
}$$
Then $\mathbf{D}_{psg}(A)=0$ if and only if $\mathbf{D}_{psg}(B) = 0 = \mathbf{D}_{psg}(C)$.
\end{theorem}

\begin{proof}
Since $\mathbf{D}^{b}_{pur}(B)$ and $\mathbf{D}^{b}_{pur}(C)$ can be fully embedded into $\mathbf{D}^{b}_{pur}(A)$, it is clear that the finiteness of $\mathrm{p.gldim}(A)$ implies the finiteness of both $\mathrm{p.gldim}(B)$ and $\mathrm{p.gldim}(C)$. By Proposition \ref{prop 4.2}, it follows that $\mathbf{D}_{psg}(B) = 0 = \mathbf{D}_{psg}(C)$ if $\mathbf{D}_{psg}(A)=0$.

For the converse, it suffices to prove that if $B$ and $C$ are of finite pure-global dimension, then $\mathrm{p.gldim}(A)<\infty$. Let $M$ and $N$ be any $A$-modules. The above recollement induces the following distinguished triangles in $\mathbf{D}^{b}_{pur}(A)$:
$$j_{!}j^{*}(M)\longrightarrow M\longrightarrow i_{*}i^{*}(M)\longrightarrow j_{!}j^{*}(M)[1],$$
$$i_{*}i^{!}(N)\longrightarrow N\longrightarrow j_{*}j^{*}(N)\longrightarrow i_{*}i^{!}(N)[1].$$
We abbreviate $\mathrm{Hom}_{\mathbf{D}^{b}_{pur}(A)}(-,-)$ with  $\mathbf{D}^{b}_{pur}(A)(-,-)$. By applying $\mathrm{Hom}_{\mathbf{D}^{b}_{pur}(A)}(-,i_{*}i^{!}(N))$ and $\mathrm{Hom}_{\mathbf{D}^{b}_{pur}(A)}(-,j_{*}j^{*}(N))$ to the first triangle, we get the following long exact sequences:
$$\cdots\rightarrow \mathbf{D}^{b}_{pur}(A)(i_{*}i^{*}(M), i_{*}i^{!}(N)[n])\rightarrow
\mathbf{D}^{b}_{pur}(A)(M,i_{*}i^{!}(N)[n])\rightarrow \mathbf{D}^{b}_{pur}(A)(j_{!}j^{*}(M),i_{*}i^{!}(N)[n])
\rightarrow \cdots $$
$$\cdots\rightarrow \mathbf{D}^{b}_{pur}(A)(i_{*}i^{*}(M), j_{*}j^{*}(N)[n])\rightarrow
\mathbf{D}^{b}_{pur}(A)(M,j_{*}j^{*}(N)[n])\rightarrow \mathbf{D}^{b}_{pur}(A)(j_{!}j^{*}(M),j_{*}j^{*}(N)[n])
\rightarrow \cdots $$
We have $\mathbf{D}^{b}_{pur}(A)(j_{!}j^{*}(M),i_{*}i^{!}(N)[n])\cong \mathbf{D}^{b}_{pur}(C)(j^{*}(M), j^{*}i_{*}i^{!}(N)[n])=0$ for every $n\in \mathbb{Z}$ since $j^{*}i_{*}=0$; and similarly, $\mathbf{D}^{b}_{pur}(A)(i_{*}i^{*}(M), j_{*}j^{*}(N)[n])=0$ follows by the adjunction $(j^{*}, j_{*})$.
Since $\mathrm{p.gldim}(B)<\infty$, $\mathbf{D}^{b}_{pur}(B)\simeq \mathbf{K}^{b}(\mathcal{PP})$. Noting that $i^{*}(M)$ and $i^{!}(N)$ lie in $\mathbf{K}^{b}(\mathcal{PP})$, one has $\mathrm{Hom}_{\mathbf{D}^{b}_{pur}(B)}(i^{*}(M), i^{!}(N)[n])= 0$ for $n>>0$; moreover, $i_{*}$ is a full embedding and then $\mathbf{D}^{b}_{pur}(A)(i_{*}i^{*}(M), i_{*}i^{!}(N)[n])=0$. Similarly, since $\mathrm{p.gldim}(C)<\infty$,
it follows that $\mathbf{D}^{b}_{pur}(A)(j_{!}j^{*}(M),j_{*}j^{*}(N)[n])=0$ for $n>>0$.

Now we apply  $\mathrm{Hom}_{\mathbf{D}^{b}_{pur}(A)}(M,-)$ to the second triangle, and obtain the following long exact sequence
$$\cdots\rightarrow \mathbf{D}^{b}_{pur}(A)(M, i_{*}i^{!}(N)[n])\rightarrow
\mathbf{D}^{b}_{pur}(A)(M, N[n])\rightarrow \mathbf{D}^{b}_{pur}(A)(M, j_{*}j^{*}(N)[n])\rightarrow \cdots $$
By the above argument, we have $\mathbf{D}^{b}_{pur}(A)(M, i_{*}i^{!}(N)[n])=0 = \mathbf{D}^{b}_{pur}(A)(M, j_{*}j^{*}(N)[n])$ for $n>>0$. Then
$\mathrm{Pext}_{A}^{n}(M,N)\cong \mathrm{Hom}_{\mathbf{D}^{b}_{pur}(A)}(M, N[n]) = 0$. This implies that $\mathrm{p.gldim}(A)<\infty$.
\end{proof}

\begin{ack*}
The authors appreciate the referee for helpful comments and suggestions, especially for pointing out \cite{Bru79, Bru83} and for some examples and counterexamples about the assumption of Theorem \ref{thm 3.1}.
This work was supported by National Natural Science Foundation of China (No. 11871125), Natural Science Foundation of Chongqing (No. cstc2018jcyjAX0541) and the Science and Technology Research Program of Chongqing Municipal Education Commission (No. KJQN201800509).\\
\end{ack*}

\end{document}